\documentclass{amsart}
\usepackage{amssymb,amsthm,amsmath,amstext,amsxtra}
\usepackage{bm}       
\usepackage{mathtools} 
\usepackage[all]{xy}
\usepackage{booktabs}
\usepackage{hyperref}
\hypersetup{colorlinks=true,urlcolor=blue,citecolor=blue,linkcolor=blue}
\usepackage{enumerate}
\usepackage{stmaryrd}
\usepackage{enumitem}
\usepackage{colonequals}
\usepackage{float} 
\usepackage[group-minimum-digits=4,group-separator=\text{,}]{siunitx}

\usepackage{wrapfig}
\usepackage[most]{tcolorbox}
\usepackage{graphicx} 
\usepackage{caption}
\usepackage{subcaption}
\newcommand{\polygons}{.}

\newtheorem{theorem}{Theorem}[section]
\newtheorem{lemma}[theorem]{Lemma}
\newtheorem{cor}[theorem]{Corollary}

\newtheorem{prop}[theorem]{Proposition}

\theoremstyle{definition}

\newtheorem{example}[theorem]{Example}

\numberwithin{equation}{section}

\newcommand{\calO}{\mathcal{O}}
\newcommand{\CC}{\mathbb{C}}
\newcommand{\FF}{\mathbb{F}}
\newcommand{\Fp}{{\FF_p}}
\newcommand{\Fq}{{\FF_q}}

\newcommand{\PP}{\mathbb{P}}
\newcommand{\QQ}{\mathbb{Q}}
\newcommand{\Qq}{{\QQ_q}}
\newcommand{\RR}{\mathbb{R}}
\newcommand{\TT}{\mathbb{T}}
\newcommand{\ZZ}{\mathbb{Z}}
\newcommand{\Zq}{{\ZZ_q}}

\newcommand{\XFq}{\mathcal{X}}
\newcommand{\YFq}{\mathcal{Y}}
\newcommand{\UFq}{\mathcal{U}}
\newcommand{\XQq}{X}
\newcommand{\YQq}{Y}
\newcommand{\UQq}{U}
\newcommand{\VQq}{V}

\newcommand{\rig}{{\operatorname{rig}}}
\newcommand{\dR}{{\operatorname{dR}}}

\DeclareMathOperator{\Fil}{Fil}
\DeclareMathOperator{\Frob}{Frob}

\DeclareMathOperator{\Gr}{Gr}

\DeclareMathOperator{\Int}{Int}

\DeclareMathOperator{\Proj}{Proj}
\DeclareMathOperator{\rank}{rank}
\DeclareMathOperator{\Spec}{Spec}

\DeclareMathOperator{\Vol}{Vol}

\newcommand{\dd}{\,\mathrm{d}}

\usepackage{xcolor}
\definecolor{darkred}{HTML}{CC1F1F}
\definecolor{green}{rgb}{.4,.7,.4}
\definecolor{blue}{rgb}{.2,.6,.75}
\definecolor{pastelb}{HTML}{3333FF}

\definecolor{pastelyellow}{rgb}{0.992157, 0.552941, 0.235294}
\definecolor{pastelorange}{rgb}{0.941176, 0.231373, 0.12549}
\definecolor{pastelred}{rgb}{0.741176, 0., 0.14902}
\definecolor{darkbrown}{rgb}{0.25098, 0., 0.0745098}

\title[Zeta functions of toric hypersurfaces]{Zeta functions of nondegenerate hypersurfaces in toric varieties via controlled reduction in $p$-adic cohomology}

\author{Edgar Costa}
\address{Department of Mathematics, Massachusetts Institute of Technology, Cambridge, MA 02139, USA}
\email{edgarc@mit.edu}
\urladdr{\url{https://edgarcosta.org}}

\author{David Harvey}
\address{School of Mathematics and Statistics, University of New South Wales, Sydney NSW
2052, Australia}
\email{d.harvey@unsw.edu.au}
\urladdr{\url{http://web.maths.unsw.edu.au/~davidharvey/}}

\author{Kiran S. Kedlaya}
\address{Univ. of California, San Diego, 9500 Gilman Drive \#0112, La Jolla, CA 92093 USA}
\email{kedlaya@ucsd.edu}
\urladdr{\url{http://kskedlaya.org/}}
\thanks{
The first author was partially supported by the Simons Collaboration Grant \#550029.
The second author was supported by the Australian Research Council (grants DP150101689 and FT160100219).
The third author was supported by NSF (grants DMS-1101343, DMS-1501214); UC San Diego (Warschawski Professorship); and a Guggenheim Fellowship.
All three authors thank ICERM for its hospitality during fall 2015.}

\begin{document}

\begin{abstract}
We give an interim report on some improvements and generalizations of the  Abbott--Kedlaya--Roe method to compute the zeta function of a nondegenerate ample hypersurface in a projectively normal toric variety over $\Fp$ in linear time in $p$.
These are illustrated with a number of examples including K3 surfaces, Calabi--Yau threefolds, and a cubic fourfold.
The latter example is a non-special cubic fourfold appearing in the Ranestad--Voisin coplanar divisor on moduli space;
this verifies that the coplanar divisor is not a Noether--Lefschetz divisor in the sense of Hassett.
\end{abstract}

\maketitle

\section{Introduction}

We consider the problem of computing the zeta function $Z(\XFq, t)$ of an explicitly specified variety $\XFq$ over a finite field $\Fq$ of characteristic $p$.
For curves and abelian varieties, Schoof's method and variants
\cite{schoof-85, pila-90, 
adleman-huang-96, gaudry-harley-00,
gaudry-schost-04, gaudry-kohel-smith-11, gaudry-schost-12}
can compute $Z(\XFq , t )$ in time and space polynomial in $\log q$ and exponential in the genus/dimension; these have only been implemented for genus/dimension at most 2.
Such methods may be characterized as \emph{$\ell$-adic}, as they access the $\ell$-adic cohomology  (for $\ell \neq p$ prime) of the variety via torsion points; there also exist \emph{$p$-adic methods}
which compute approximations of the Frobenius action on $p$-adic cohomology (Monsky--Washnitzer cohomology),
and which have proven to be more viable in practice for large genus.
Early examples include Kedlaya's algorithm \cite{kedlaya-01} for hyperelliptic curves,
in which the time/space dependence is polynomial in the genus and quasi-linear in $p$,
and Harvey's algorithm \cite{harvey-07} which improves the dependence on $p$
to $p^{1/2+\epsilon}$.
These methods have been subsequently generalized \cite{gaudry-gurel-01, denef-vercauteren-06, denef-vercauteren-06b, harrison-12},
notably by Tuitman's algorithm \cite{tuitman-16, tuitman-17} which applies to (almost) all curves while keeping the quasi-linear dependence on $p$.
In another direction, Harvey \cite{harvey-14} has shown that when  computing the zeta functions of reductions of a fixed hyperelliptic curve over a number field, $p$-adic methods can achieve \emph{average} polynomial time in $\log p$
and the genus; this has been implemented in small genus \cite{harvey-sutherland-14a, harvey-sutherland-16}.

One advantage of $p$-adic methods over $\ell$-adic ones is that they scale much better to higher-dimensional varieties.
For example, there are several $p$-adic constructions that apply to \emph{arbitrary} varieties with reasonable
asymptotic complexity \cite{lauder-wan-08, harvey-15}, although we are not aware of any practical implementations.
Various algorithms, and some implementations, have been given using Lauder's \emph{deformation method}
of computing the Frobenius action on the Gauss--Manin connection of a pencil
 \cite{lauder-04, lauder-04b, gerkmann-07, hubrechts-07, 
 hubrechts-10, kedlaya-13,
 pancratz-tuitman-15, tuitman-18}.

In this paper, we build on an algorithm of Abbott--Kedlaya--Roe \cite{abbott-kedlaya-roe-10}
which adapts the original approach of \cite{kedlaya-01} to smooth projective hypersurfaces. Here, we add two key improvements.
\begin{itemize}
\item
We use \emph{controlled reduction} in de Rham cohomology, as described in some lectures of Harvey \cite{harvey-hypersurface1, harvey-hypersurface2, harvey-hypersurface3}, to preserve sparsity of certain polynomials, thus reducing the time (respectively, space) dependence on $p$ from polynomial to quasi-linear (respectively, $O(\log p)$).
The resulting \emph{controlled AKR method} was implemented, with further improvements, in Costa's Ph.D. thesis \cite{costa-thesis}, with examples of generic surfaces and threefolds over $\FF_p$ for $p\sim10^6$ \cite[Section 1.6]{costa-thesis}; by contrast, the largest $p$ used in \cite{abbott-kedlaya-roe-10} is $29$.
Costa and Harvey are currently preparing a paper on this method; meanwhile, Costa's GPL-licensed code is available on \textsc{GitHub} \cite{controlledreduction}, and is slated to be integrated into \textsc{SageMath} \cite{sage}.
\item
We also generalize to toric hypersurfaces, subject to a standard genericity condition called \emph{nondegeneracy}.
This greatly increases the applicability of the method while preserving much of its efficiency. Some previous attempts have been made to compute zeta functions in this setting, such as work of Castryck--Denef--Vercauteren \cite{castryck-denef-vercauteren-06} for curves and Sperber--Voight \cite{sperber-voight-13} in general; it is the combination 
with controlled reduction that makes our approach the most practical to date.
\end{itemize}
It may be possible to improve the dependence on $p$ to square-root (as in \cite{harvey-07}) or average polynomial time
(as in \cite{harvey-14}), but we do not attempt to do so here.

For reasons of space, we give only a summary of the algorithm, with further details to appear elsewhere.
In lieu of these details, we present a number of worked examples in dimensions 2-4
that demonstrate the practicality of this algorithm in a wide range of cases.
The results are based on an implementation in \texttt{C++}, using \texttt{NTL} \cite{ntl} for the underlying arithmetic operations.
Our examples in dimensions 2 and 3 were computed on
one core of a desktop machine with an \texttt{Intel(R) Core(TM) i5-4590 CPU @ 3.30GHz};
our sole example in dimension 4 was computed on one core of a server with an \texttt{AMD Opteron Processor 6378 @ 1.6GHz}.
(We have not yet optimized our vector-matrix multiplications in any way;
as a consequence, we observe a serious performance hit whenever the working moduli exceeds $2^{62}$.)

In dimensions 2 and 3, our examples are  \emph{Calabi--Yau varieties}, i.e., smooth, proper, simply connected varieties with trivial canonical bundle. In dimension 1, these are simply elliptic curves. In dimension 2, they are \emph{K3 surfaces}, 
whose zeta functions are of computational interest for various reasons.
For instance, these zeta functions can (potentially) be used to establish the infinitude of rational curves on a K3 surface (see the introduction to \cite{costa-tschinkel-14} for discussion); there has also been recent work on analogues of the Honda--Tate theorem, establishing conditions under which particular zeta functions are realized by K3 surfaces \cite{taelman-16, ito-16}.

As for Calabi--Yau threefolds, much of the interest in their zeta functions can be traced back to \emph{mirror symmetry} in mathematical physics. An early example is the work of Candelas--de la Ossa--Rodriguez Villegas \cite{candelas-ossa-villegas-01} on the Dwork pencil; a more recent example is \cite{doran-kelly-salerno-sperber-voight-whitcher-16},
in which (using $p$-adic cohomology) certain mirror families of Calabi--Yau threefolds are shown to have related zeta functions.

Our four-dimensional example is a cubic projective fourfold. Such varieties occupy a boundary region between rational and irrational varieties; it is expected that a rational cubic fourfold is \emph{special} in the sense of having a primitive cycle class in codimension 2. The geometry of special cubic fourfolds is in turn closely linked to that of K3 surfaces; in many cases, the Hodge structure of a K3 surface occurs (up to a twist) inside the Hodge structure of a special cubic fourfold, and (modulo standard conjectures) this implies a similar relationship between zeta functions. See \cite{hassett-16} for further discussion.

The specific example we consider is related to the geometry of the moduli space of cubic fourfolds over $\CC$. On this space, there exist various divisors consisting entirely of special cubic fourfolds; Hassett
calls these \emph{Noether--Lefschetz divisors} (by analogy with the case of surfaces). Recently, Ranestad--Voisin \cite{ranestad-voisin-17}
exhibited four divisors which they believed not to be Noether--Lefschetz, but only checked this in one case. Addington--Auel \cite{addington-auel} checked two more cases by finding in these divisors some cubic fourfolds over $\QQ$ with good reduction at 2, such that the zeta functions over $\FF_2$ show no primitive Tate classes in codimension 2.
By replacing the brute-force point counts of Addington--Auel with $p$-adic methods, we are able to work modulo a larger prime to find an example showing that the fourth Ranestad--Voisin divisor is not Noether--Lefschetz.

To sum up, the overall goal of this project is to vastly enlarge the collection of varieties for which computing the zeta function is practical. It is our hope that doing so will lead to a rash of new insights, conjectures, and theorems of interest to a broad range of number theorists and algebraic geometers.

\section{Toric hypersurfaces}

We begin by reviewing the construction of a projective toric variety from a lattice polytope.
For more details we recommend \cite{cox-little-schenck-11}.

Let $n \geq 1$ be an integer.
For any commutative ring $R$, let $R[x^{\pm}]$ denote the Laurent polynomial ring in $n$ variables $x_1, \dots, x_n$ with coefficients in $R$.
For $\alpha  \colonequals (\alpha_i) \in \ZZ^n$, we write $x^\alpha$ for the monomial $x_1 ^{\alpha_1} \cdots x_n ^{\alpha_n}$.
We denote the $R$-torus by $\TT^n _R := \Spec(R[x^{\pm}])$.

Let $\Delta \subset \RR^n$ be the convex hull of a finite subset of $\ZZ^n$ that is not contained in any hyperplane,
so that $\dim \Delta = n$.
For $r \in \RR$, let $r \Delta$ be the $r$-fold dilation of $\Delta$.
For an integer $d \geq 0$, let
\[
P_d \colonequals \langle   x^\alpha : \alpha \in d \Delta \cap \ZZ^n \rangle_R \quad (\mbox{resp.\,} P^{\Int} _d \colonequals \langle   x^\alpha : \alpha \in \Int(d \Delta) \cap \ZZ^n \rangle_R)
\]
be the free $R$-module on the set of monomials with exponents in $d \Delta \cap \ZZ^n$ (resp.  $\Int(d \Delta) \cap \ZZ^n$).
Define the $R$-graded algebras
\begin{equation*}
  P_\Delta := \bigoplus_{d = 0} ^{+\infty} P_d  \quad \text{and} \quad P^{\Int} _\Delta := \bigoplus_{d = 0} ^{+\infty} P^{\Int}_d .
\end{equation*}
with the usual multiplication in $R[x^{\pm}]$.
We define the polarized toric variety associated to $\Delta$ as the pair $(\PP_\Delta, \calO_\Delta)$, where  $\PP_\Delta \colonequals \Proj P_\Delta$ and $\calO_\Delta$ is the ample line bundle on $\PP_\Delta$ associated to the graded $P_\Delta$-module $P_\Delta (1)$.
Note that $P_\Delta$ and $P^{\Int} _\Delta$ admit $n$ commuting degree-preserving differential operators $\partial_i \colonequals x_i \frac{\partial}{\partial x_i}$ for $i = 1, \dots, n$.

In order to suppress some expository and algorithmic complexity,
we make the simplifying assumption that $\Delta$ is a \emph{normal} polytope; that is, the map
\begin{equation*}
  (\Delta \cap \ZZ^n)^d \to d\Delta \cap \ZZ^n:
  (x_1,\dots,x_d) \mapsto x_1 + \cdots + x_d
\end{equation*}
is surjective for $d \geq 1$.
This corresponds to the pair $(\PP_\Delta, \calO_\Delta)$ being
\emph{projectively normal}; this will be the case in our examples.
As a consequence, we have that $\calO_\Delta$ is indeed very ample.

\begin{example}
  Let $\Delta$ be the regular $n$-simplex, the convex hull of $0, e_1, \dots, e_n $.
We may then identify $P_{d}$ with the set of homogeneous polynomials of degree $d$ in $x_0,\dots,x_n$,
  by identifying  $x^\alpha \in P_{\Delta,d}$ with  the monomial $x_0^{d-\alpha_1-\cdots-\alpha_n} x_1^{\alpha_1} \cdots x_n^{\alpha_n}$;
  then $(\PP_\Delta, \calO_\Delta)$ is isomorphic to $(\PP^n _{R}, \calO(1))$.

  We obtain the weighted projective space $\PP(w_0, \dots, w_n)$ by taking
  $$\Delta = \{(x_0, \dots, x_n) \in \RR^{n + 1} : \sum_{i = 0}^n w_i x_i = w_0 \cdots w_n \}, \quad \text{see \cite[1.2.5]{dolgachev-82}}.$$

  We obtain $\PP^k_R \times_R \PP^r_R$ by taking $\Delta$ to be the Cartesian product of the regular $k$-simplex by the regular $r$-simplex  \cite[\S 2.4]{cox-little-schenck-11}.
\end{example}

We now turn our attention to toric hypersurfaces over $R = \FF_q$, the finite field with $q = p^a$ elements and characteristic $p$.
Let $\YFq$ be the hypersurface in $\TT^n _{\FF_q}$ defined by a Laurent polynomial $\overline{f} \in  \FF_q[x^{\pm}]$,  $\YFq \colonequals V(\overline{f}) \subset \TT^n _{\Fq}$.
Let
\begin{equation*}
 \operatorname{supp}{\overline{f}} = \{\alpha \in \ZZ^n : \overline{c}_\alpha \neq 0 \}
\end{equation*}
be the support of $\overline{f}$ in $\RR^n$; the convex hull of $\operatorname{supp}{\overline{f}}$
is the \emph{Newton polytope} of $\overline{f}$, which we denote by $\Delta$.
We will work under the hypothesis that $\overline{f}$ is \emph{$(\Delta-)$non\-degenerate}:\footnote{This condition was introduced by Dwork \cite{dwork-62} without a name;
the term \emph{nondegenerate} first appears in \cite{kouchnirenko, varchenko}. Synonyms include \emph{$\Delta$-regular} \cite[\S 4]{batyrev-93} and \emph{sch\"on} \cite{tevelev}.}
for all faces $\tau \subseteq \Delta$ (including $\Delta$ itself), the system of equations
\begin{equation*}
  \overline{f}_{\restriction \tau} = \partial_1 \overline{f}_{\restriction \tau} = \cdots = \partial_n \overline{f}_{\restriction \tau} = 0
\end{equation*}
has no solution in ${\overline{\FF}_q^{\times n}}$, where $\overline{\FF}_q$ denotes an algebraic closure of $\FF_q$.
Furthermore, nondegeneracy implies quasi-smoothness, see \cite[Definition 3.1 and Proposition 4.15]{batyrev-cox-94}.
For fixed normal $\Delta$ over an infinite field, this condition holds for generic $\overline{f}$.
Others have given point-counting algorithms under this assumption \cite{castryck-denef-vercauteren-06, sperber-voight-13}.

Let $\XFq \colonequals \Proj P_\Delta/\bigl(\overline{f}\bigr)$ denote the closure of $\YFq$ in $\PP_\Delta$
(placing $\overline{f}$ in degree 1)
and set $\UFq \colonequals \TT^n \backslash \YFq$.
Let $H^i _\rig$ denote the $i$th rigid cohomology group in the sense of Berthelot \cite{berthelot-97}.
The Lefschetz hyperplane theorem, combined with Poincar\'{e} duality, show that the map
\begin{equation*}
  H^{i} _\rig (\PP _\Delta ) \rightarrow H^{i} _\rig (\XFq ),
\end{equation*}
induced by the inclusion $\XFq \hookrightarrow \PP_\Delta$ is an isomorphism for $i \neq n - 1$ \cite[10.8]{batyrev-cox-94}.
This implies that the ``interesting'' part of the cohomology of $\XFq$ occurs in dimension $n - 1$ and consists of those classes that do not come from $P_\Delta$.
Denote by $PH^{n-1} _\rig (\XFq)$ the primitive cohomology group of $X$, defined by the (Frobenius-equivariant) exact sequence
\begin{equation*}
  0 \rightarrow H^{n-1} _\rig (\PP_\Delta) \rightarrow H^{n-1} _\rig (\XFq) \rightarrow PH^{n-1} _\rig (\XFq) \rightarrow 0
\end{equation*}
With this notation, we may write
\begin{equation*}
  Z(\XFq, t)  = Z(\PP_\Delta,  t) Q(t)^{(-1)^n}.
\end{equation*}
where
\begin{equation*}
  Q(t) \colonequals \det \bigl(1 - t \Frob_q | PH^{n-1} _\rig (\XFq) \bigr).
\end{equation*}
Thus given $\overline{f}$, we would like to compute $Q(t)$.

The cohomology group $PH^{n-1} _\rig (\XFq)$ is closely related to $H^{n} _\rig (\PP_\Delta \backslash \XFq)$.
For example, if $\PP_\Delta$ is a (weighted) projective space, as in \cite{abbott-kedlaya-roe-10} and \cite{costa-thesis}, the two cohomology groups are isomorphic; see \cite[10.11]{batyrev-cox-94}.

\section{de Rham cohomology of toric hypersurfaces}
In preparation for our use of $p$-adic cohomology to compute $Q(t)$, we give an explicit description of the algebraic de Rham cohomology of a nondegenerate toric hypersurface in characteristic zero. We take $R$ to be the ring $\Zq$, the ring of integers of $\Qq$, the unramified extension of $\QQ_p$ with residue field $\Fq$.

Let $f \in \Zq[x^\pm]$ be a lift of $\overline{f}$ to characteristic zero with the same support as $f$ (it will also be nondegenerate).
Consider $\YQq \colonequals V(f) \subset \TT \colonequals \TT _{\Qq}$ and $\XQq$, the closure of $\YQq$ in $\PP_\Delta$.
Write $\UQq \colonequals \TT \backslash \YQq$, and $\VQq \colonequals \PP_\Delta \backslash \XQq \simeq \Spec(A)$, where $A$ is the coordinate ring of $\VQq$; explicitly,
\begin{equation*}
  A \simeq \bigcup _{d = 0}^{+\infty} f^{-d} P_d.
\end{equation*}
Let $I_f$ be the ideal in $P_\Delta$ generated by $f, \partial_1 f, \dots, \partial_n f$.
We call $I_f$ the \emph{toric Jacobian ideal} and the quotient ring $J_f \colonequals P_\Delta/ I_f$ the \emph{toric Jacobian ring}.
Since $f$ is nondegenerate, the ideal $I_f$ is irrelevant in $P_\Delta$ and $\rank_{\Zq} J_f = n! \Vol(\Delta)$; furthermore, $(J_{f}) _d = 0$ for $d > n$ \cite[\S 4]{batyrev-93}.
If $\calO_\Delta$ is not very ample, then $I_f$ might not be generated in degree $1$ and we might have $(J_{f}) _d = 0$ only for $d \gg n$.

Let $\Omega^\bullet$ denote the logarithmic de Rham complex of $\VQq$ with poles along $\PP_{\Delta} \backslash \TT$.
Let $H^{\bullet}$ be the cohomology groups of $\Omega^\bullet$; these are naturally isomorphic to $H^{\bullet} _\dR( \VQq \cap \TT = \TT \backslash Y = \UQq)$ and $H^{\bullet} _\rig (\TT_{\Fq} \backslash \YFq = \UFq)$ \cite{kato-89}.

We now provide an explicit description of the group $H^n$, as in \cite[\S 6 and 7]{batyrev-93}, in which we will compute $Q(t)$.
Set
$$\omega \colonequals \frac{\dd x_1}{x_1} \wedge \cdots \wedge \frac{\dd x_n}{x_n} \in \Omega^{n},$$
and define the ascending filtration in $\Omega^n$ by
\begin{equation*}
  \Fil ^d \Omega^n \colonequals \left\{ g f^{-d} \omega : g \in P_d \right\}.
\end{equation*}
The associated graded ring
\[
  \Omega^n := \bigoplus_{d=0}^\infty \Gr^d \Omega^n, \qquad
  \Gr^d \Omega^n := \Fil^d \Omega^n / \Fil^{d-1} \Omega^n
\]
is then isomorphic to $P_\Delta/(f)$ (again placing $f$ in degree 1).

Equip $H^n$ with the filtration induced from $\Omega^n $,
and view $H^n$ as the quotient
of $\Omega^n$ by the $\Qq$-submodule generated by the relations
\begin{equation}
  \label{eq:relations}
  \frac{g}{f^d} \omega - \frac{g f}{f^{d+1}} \omega 
  \quad \text{and} \quad
  \frac{\partial_i(g)}{f^d} \omega -  \frac{d g \partial_i(f)}{f^{d+1}} \omega
\end{equation}
for each $i = 1, \dots, n$, each nonnegative integer $d$, and each $g \in P_{d}$.
From these relations, we see that
\begin{equation*}
\Gr^1 H^n \simeq P_1/(f) \quad \text{and} \quad \Gr^d H^n \simeq (J_f)_d \qquad (d > 1).
\end{equation*}
This gives a way to compute explicitly in $H^n$: for any $h \in (J_f) _{d + 1}$ with $d>n$, we can find a relation of the form
\begin{equation}
  \label{eq:griffithsreduction}
  d \frac{h}{ f^{\, d + 1}} \omega = d \frac{g_0 f + \sum_{i = 1} ^n g_i  \partial_i f}{f^{ d + 1}} \omega \equiv \frac{ d g_0 + \sum_{i = 1}^n \partial_i g_i  }{f^{ d }} \omega.
\end{equation}
because  $P_d \subset (I_f)_d$, so in $H^n$ we can reduce the pole order of any form to at most $n$.
This process was introduced for smooth projective hypersurfaces in \cite{griffiths-69} and attributed to Dwork;
it is commonly known as \emph{Griffiths--Dwork reduction}.

With the above representation of $H^n$, we may also identify $PH^{n-1} _\dR  (X)$ with $(P^{\Int} _\Delta + I_f)/I_f \subset H^n$, where the filtration by pole order is the Hodge filtration; see \cite[\S 9, \S 11]{batyrev-93, batyrev-cox-94}.

We now introduce a variation of Griffiths--Dwork reduction, called \emph{controlled reduction}.
This will be crucial for our application to $p$-adic cohomology, as careless application of Griffiths--Dwork reduction to a sparse form will easily lead to a dense form.
For $d = 1, \dots, n + 1$, choose a $\Zq$-linear splitting $P_d \approx (J' _f)_d  \oplus C_d$,
where $(J' _f)_d$ is a lift of $(J _f)_d$ into $P_d$.
Let $\rho_d \colon P_d \rightarrow (J'_f)_d$ and $\pi_{d,0}, \dots,  \pi_{d,n} \colon P_d \rightarrow P_{d-1}$ be $\Zq$-linear maps such that
\begin{equation*}
  g = \rho_d( g ) + \pi_{d,0}( g ) \cdot f + \sum_{i = 1} ^n \pi_{d,i} (g) \cdot \partial_i f ; \quad g \in P_d.
\end{equation*}
These maps may be constructed one monomial at a time.

\begin{prop}[Controlled reduction] \label{P:controlled reduction}
  Let  $x^\nu \in P_{1}$ and  $x^\mu \in P_{d}$ be two monomials and
  define the following $\Zq$-linear maps:
  \begin{equation*}
    \begin{aligned}
      R_{\mu, \nu}(g) &:= (d+n) \pi_{n+1,0}(x^\nu g) + \sum_{i=1}^n (\partial_i + \mu_i)(\pi_{n+1,i}(x^\nu g))
      \\
      S_{\nu}(g) &:= \pi_{n+1,0}(x^\nu g) + \sum_{i=1}^n \nu_i \pi_{n+1,i}(x^\nu g) 
    \end{aligned}
  \end{equation*}
  Then for any $g \in P_{n}$ and any nonnegative integer $j$, in $H^n$ we have
  \[
    g \frac{ x^{(j+1) \nu + \mu} }{f^{d+n+j+1}} \omega\equiv (d+n+j)^{-1} (R_{\mu, \nu}(g) + j S_{\nu}(g) ) \frac{ x^{j \nu + \mu} }{f^{d+n+j}} \omega.
  \]
  \end{prop}
\begin{proof}
  This is straightforward from \eqref{eq:relations} and \eqref{eq:griffithsreduction}.
\end{proof}
Note that Proposition~\ref{P:controlled reduction} enables us to reduce the pole order of a differential form from $d+n+j+1$ to $d+n+j$ without increasing its total number of monomials; we can thus reduce the pole order of a sparse form without making it dense.
\begin{cor} \label{C:controlled reduction}
  With notation as in Proposition~\ref{P:controlled reduction}, let $k$ be  a positive integer. Then for any $g \in P_{n}$,
  \begin{equation*}
    g \frac{ x^{\mu + k \nu } }{f^{d+n+k}} \omega\equiv \frac{\prod_{j=0}^{k-1} (R_{\mu,\nu} + j S_{\nu})(g) }{\prod_{j=0}^{k-1} (d+n+j)}
    \frac{x^\mu}{f^{d+n}} \omega,
  \end{equation*}
  forming the composition product from left to right.
\end{cor}

Using Proposition~\ref{P:controlled reduction} amounts to performing linear algebra on matrices of size $\#(n\Delta \cap \ZZ^n) \sim n^n \Vol(\Delta)$.
One can reduce this by a factor of $n^n/n! \sim e^n$ at the expense of making the expression for the reduction matrix more convoluted;
compare \cite[Proposition 1.17 and 1.18]{costa-thesis}.

\section{Monsky--Washnitzer cohomology}

We now indicate how Monsky--Washnitzer cohomology, as introduced in \cite{monsky-washnitzer-68a, monsky-68b, monsky-68c}, provides a crucial link between algebraic de Rham cohomology and $p$-adic rigid cohomology,
by transferring to the former the canonical Frobenius action on the latter; see also \cite{vanderput-86}.
To simplify, we assume $p > \max\{n, 2\}$.

Let $A^\dagger$ denote the \emph{weak $p$-adic completion} of $A$, the ring consisting of formal sums $\sum_{d =0} ^{+\infty} g_d f^{-d}$ such that for some $a, b > 0$,
$g_d \in p^{\max\{0, \lfloor ad-b \rfloor\}} P_d$ for all $d \geq 0$.
We define the associated logarithmic de Rham complex $\Omega^{\dagger, \bullet}$ by
$\Omega^{\dagger, i} \colonequals \Omega^i \otimes_{A} A^\dagger;$
denote the cohomology groups of this complex by $H^{\dagger, \bullet}$.
We may then obtain $p$-adic Monsky--Washnitzer cohomology groups $H^{\dagger, \bullet} \otimes_\Zq \Qq$.
The map $\Omega^\bullet \otimes_{\Zq} \Qq \to \Omega^{\dagger,\bullet} \otimes_\Zq \Qq$ is a quasi-isomorphism \cite{monsky-70, vanderput-86,kato-89};
that is, the induced maps $H^i \otimes_{\Zq} \Qq \to H^{\dagger,i} \otimes_\Zq \Qq$ are isomorphisms.
We can thus identify the algebraic de Rham cohomology of $\UQq$ with the Monsky--Washnitzer cohomology of $\UFq$.

On the other hand, we also have
$ H^{\dagger,\bullet} \otimes_{\Zq} \Qq \simeq H^{\bullet} _\rig (\UFq)$
and the latter object is functorial with respect to geometry in characteristic $p$ \cite{berthelot-97}.
In this way, $H^{\dagger,i}$ receives an action of the Frobenius automorphism, which we can make explicit by constructing a lift $\sigma$ of the $p$-th power Frobenius on $\Fq$ to $A^\dagger$.
To do so, we take the Witt vector Frobenius on $\ZZ_q$ and
set $\sigma(\mu) = \mu^p$ for any monomial  $\mu \in P_\Delta$.
We then extend $\sigma$ to $A^\dagger$ by the formula
\begin{equation}
  \label{eqn:sigmadense}
  \sigma\left( \frac{g}{f^d} \right) := \sigma(g)\sigma(f)^{-d} = \sigma(g) \sum_{i \geq 0} {-d \choose i} \frac{(\sigma(f) - f^p)^i}{f^{p(d+i)}}.
\end{equation}
The above series converges (because $p$ divides $\sigma(f) - f^p$) and the definitions ensure that $\sigma$ is a semilinear  (with respect to the Witt vector Frobenius) endomorphism of $A^\dagger$.
We finally extend $\sigma$ to $\Omega^{\dagger,\bullet}$ by $\sigma(g \dd h) := \sigma(g) \dd(\sigma(h))$.

\section{Sketch of the algorithm}
\label{section:sketch}

We now indicate briefly how to use controlled reduction to compute the Frobenius action on the cohomology of nondegenerate toric hypersurfaces.
We start as in \cite[Proposition~4.1]{harvey-07}, by rewriting the Frobenius action in a sparser form.

\begin{lemma} \label{L:congruence}
  For any positive integers $d,N$ and $g \in P_d$, in $A^\dagger$ we have
  \[
    \sigma \left(
    \frac{g}{f^d} \right)
    \equiv
    \sum_{j=0}^{N-1} \binom{-d}{j} \binom{d+N-1}{d+j}  \sigma(g f^j) f^{-p(d+j)} \pmod{p^N}.
  \]
\end{lemma}
\begin{proof}
  This follows from \eqref{eqn:sigmadense} by truncating the sum and then rewriting it formally; see \cite[Lemma 1.10]{costa-thesis}.
\end{proof}

In order to compute a $p$-adic approximation of the Frobenius action on $PH^{n-1}(\XFq)$,
we must first fix a basis of the latter; we do this by constructing a monomial basis for $PH^{n-1} _\dR (\XQq)$
via explicit linear algebra.
We then apply Frobenius to each basis element in the sparse truncated form given by Lemma~\ref{L:congruence}; recursively reduce the pole order using Corollary~\ref{C:controlled reduction}
(using $k=p$ as much as possible); and project to the chosen monomial basis.
The dominant step is controlled reduction,
which amounts to $O(p n^N \Vol(\Delta) )$ matrix multiplications of size $n! \Vol(\Delta)$ per basis element.

We will not address precision estimates in this report, except to note that the machinery of \cite[\S 3.4]{abbott-kedlaya-roe-10} applies.
In general, if we want $N$ digits of $p$-adic accuracy, we must apply Lemma~\ref{L:congruence} with $N$ replaced by $N' = N + O(n+ \log N)$ and work modulo $p^{O(N')}$.
Hence, with respect to $p$ alone, we expect our algorithm to run in quasi-linear time in $p$ and use $O(\log p)$ space.

\section{K3 surfaces}

We now turn our attention to examples, starting with K3 surfaces.
For $X$ a K3 surface, $\dim H^2(X) = 22$ and the Hodge numbers are $(1, 20, 1)$.
A common example of a K3 surface is a smooth quartic surface in $\PP^3$;
however, they also occur in other ways, such as hypersurfaces in weighted projective spaces.
Using a criterion of Miles Reid \cite{reid-80}, Yonemura \cite{yonemura-90} found the complete list of (polarized) weighted projective spaces in which a generic hypersurface is a K3 surface; there are 95 of these.
For toric varieties, the corresponding classification is that of reflexive 3-dimensional polytopes, of which there are $\num{4319}$
in all \cite{kreuzer-skarke-98}. 

In the following examples, we worked modulo $p^4$ in order to obtain $Q(t)$ with $2$ $p$-adic significant digits.
As a result, we observe a performance hit for $p > 2^{16}$.

\begin{example}
  \label{example:k3dwork}
  Consider the projective quartic surface $\XFq \subset \PP^3 _{\FF_p}$ defined by
  $$x ^4 + y ^4 + z ^4 + w ^4 + \lambda x y z w = 0;$$
  it is a member of the Dwork pencil. For $p = 2^{20} -3$ and $\lambda = 1$, using the \emph{controlled AKR algorithm} in \texttt{22h7m} we compute that
  $$ Z(\XFq, t)^{-1} = (1 - t) (1 - p t)^{16}  (1 + p t)^{3}  (1 - p^2 t) Q(t),$$
  where the ``interesting'' factor is
  $$Q(t) = (1 + p t )(1 - 1688538 t + p^2 t^2).$$
  For this family, the remaining factors, apart from $Q(t)$, could have also been deduced by a $p$-adic formula of de la Ossa–-Kadir \cite[Chapter 6]{kadir-thesis}.
  In this context,  the Hodge numbers of $PH^{2}(\XFq)$ are $(1,19,1)$.
  {\makeatletter
    \let\par\@@par
    \par\parshape0
    \everypar{}
    \begin{wrapfigure}{r}{0.32\textwidth}
      \vspace{-1em}
      \tcbox[colback=white, top=2pt,left=0pt,right=0pt,bottom=2pt, boxrule=0.2mm, on line]{
        \includegraphics[width=0.28\textwidth]{\polygons/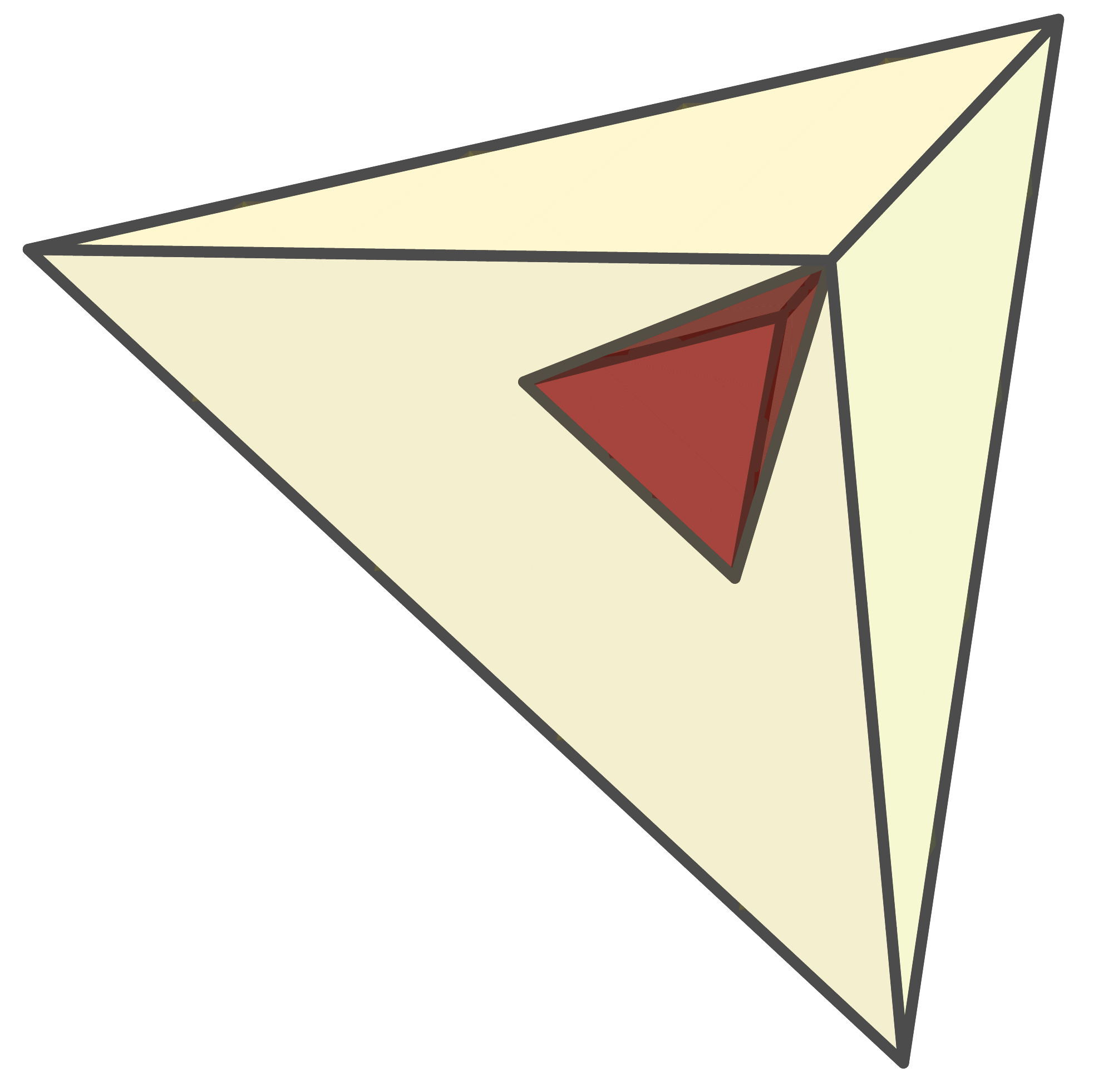}
      }
    \end{wrapfigure}
    A similar runtime would be expected if we used our current implementation to compute $Z(\XFq, t)$ with $\Delta$ being the 3-simplex (tetrahedron), as indicated by the outer polytope at right.
    Instead, we observe that the monomials defining $\XFq$ generate a sublattice of index $4^2$ in $\ZZ ^3$;
    hence, we can instead run our algorithm with a polytope of significantly smaller volume ($32/3 \approx 10.66$ versus $2/3 \approx 0.66$), as indicated by the inner polytope at right. This leads to a dramatic speedup:
    with our current implementation, we computed $Q(t)$ in \texttt{1m33s}.
  \par}%
    \noindent
    We present the running times for other $p$ in Table~\ref{table:k3dwork}; memory usage is about 16MB.

    In the new framework, $\XFq$ is given by the closure (in $\PP_\Delta$) of the affine surface defined by the Laurent polynomial
    \begin{equation*}
      x^{4} y^{-1} z^{-1} + \lambda x + y + z + 1,
    \end{equation*}
    and the Hodge numbers of $PH^{2}(\XFq)$ are $(1,1,1)$, which explains why $\deg Q(t) = 3$.

    Since the Dwork pencil is a ``small'' deformation of the Fermat quartic, we may also use
    the Pancratz--Tuitman implementation of the \emph{deformation method} \cite{pancratz-tuitman-15}
    to compute $Z(\XFq, t)$.
    We did this and verified that our results agree; we compare running times in Table~\ref{table:k3dworkdeformation}.
    To interpret these fairly, note that Pancratz--Tuitman work in $\PP^3$ and so compute the whole numerator of $Z(\XFq, t)$ rather than just $Q(t)$. 
(Note that the algorithm of \cite{tuitman-18} has a square-root dependence on $p$,
    as in \cite{harvey-07}.)

    \begin{table}[h]
      \begin{tabular}{lrr|lr}
        $p$ & CHK time & PT time
            &
        $p$ & CHK time \\
        \hline
        $2^{8} - 5$ & \texttt{0.03s} & \texttt{1.65s}
                    &
        $2^{17} - 1$  &  \texttt{11.9s}
        \\
        $2^{9} - 3$ & \texttt{0.04s} & \texttt{3.64s}
                    &
        $2^{18} - 5$ & \texttt{23.4s}
        \\
        $2^{10} - 3$ & \texttt{0.04s} & \texttt{7.39s}
                     &
        $2^{19} - 1$ & \texttt{46.9s}
        \\
        $2^{11} - 9$ & \texttt{0.06s} & \texttt{14.65s}
                     &
        $2^{20} - 3$ & \texttt{1m33s}
        \\
        $2^{12} - 3$ & \texttt{0.08s} & \texttt{34.80s}
                     &
        $2^{21} - 9$ & \texttt{3m6s}
        \\
        $2^{13} - 1$  &  \texttt{0.13s} & \texttt{34.80s}
                      &
        $2^{22} -3 $ & \texttt{6m15s}

        \\
        $2^{14} - 3$  &  \texttt{0.22s} & \texttt{2m33s}

        \\
        $2^{15} - 19$ &  \texttt{0.41s} & \texttt{6m43s}

        \\
        $2^{16} - 15$ &  \texttt{5.72s} & \texttt{14m14s}
        
      \end{tabular}
      \caption{The second and fifth columns use our current implementation to compute $Q(t)$.
        The third column uses the Pancratz--Tuitman implementation \cite{pancratz-tuitman-15} to compute $Z(\XFq,t)$.
      }
    \label{table:k3dwork}
    \label{table:k3dworkdeformation}
    \end{table}
\end{example}

\vspace{-24pt}
\begin{example}
  \label{example:c2f2}
  Consider the projective quartic surface $\XFq \subset \PP^3 _{\FF_p}$ defined by
  $$x^3 y + y^4 + z^4 + w^4 - 12  x y z w;$$
  it contains a hypergeometric motive
  (see \cite[Section 5]{doran-kelly-salerno-sperber-voight-whitcher-16}).
For $p = 2^{15} - 19$, using the \emph{controlled AKR algorithm} in \texttt{27m12s} we compute that
\begin{multline*}
  Z(\XFq, t)^{-1} = (1 - t)  (1 - p t)^{2}  (1 +  p t)^{2}  (1 - p t + p^2 t^2)^2
  (1 - p^2 t^2 + p^4 t^4)^2  ( 1 - p^2 t) Q(t),
\end{multline*}
where the ``interesting'' factor is (up to rescaling)
$$
p Q(t /p ) = p  + 20508t^{1} - 18468t^{2} - 26378t^{3} - 18468t^{4} + 20508t^{5} + p t^6.
$$
{\makeatletter
  \let\par\@@par
  \par\parshape0
  \everypar{}
  \begin{wrapfigure}{r}{0.39\textwidth}
    \vspace{-0.5em}
    \tcbox[colback=white, top=2pt,left=0pt,right=0pt,bottom=2pt, boxrule=0.2mm, on line]{
      \includegraphics[width=0.36\textwidth]{\polygons/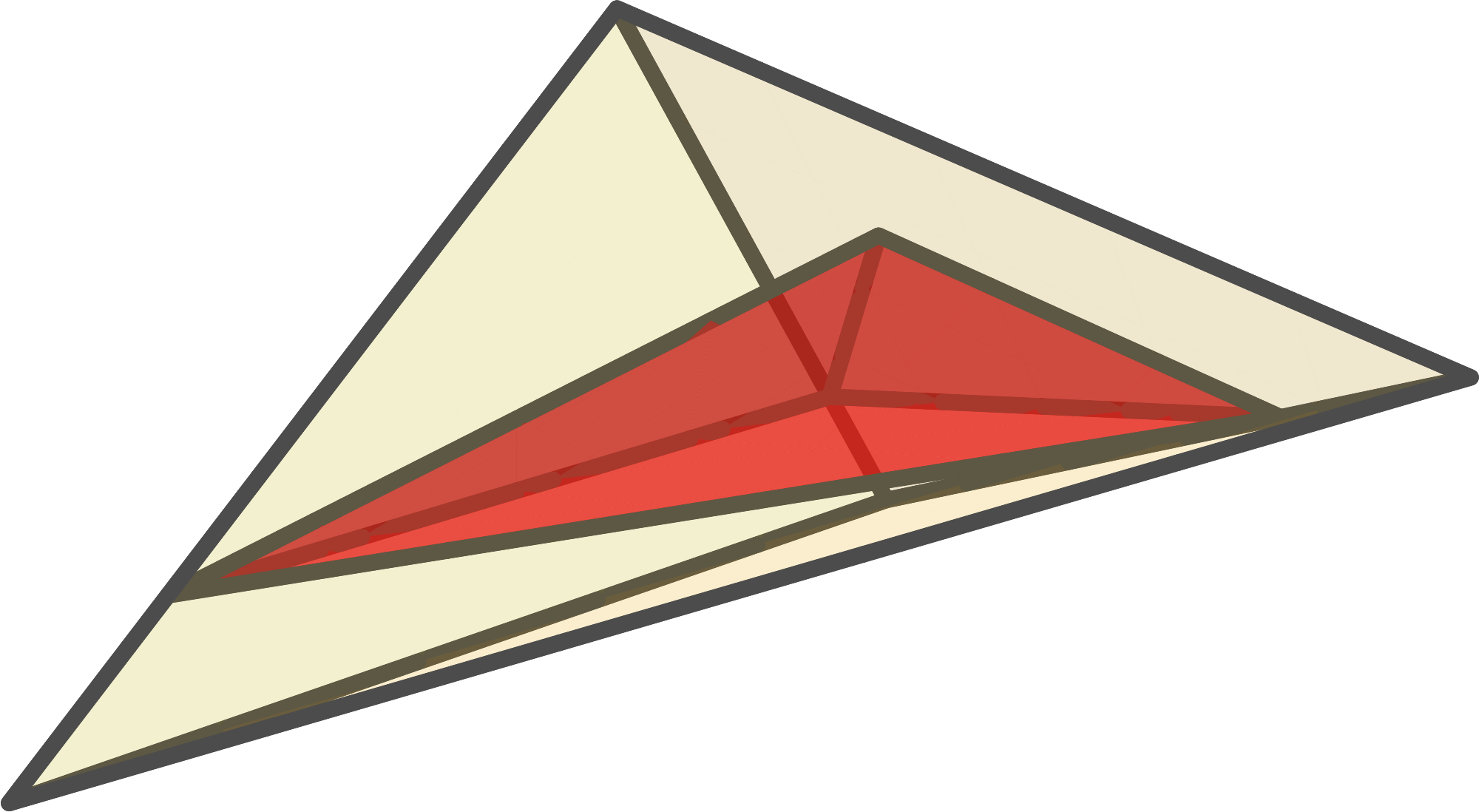}
    }
  \end{wrapfigure}
  As in the previous example, the Newton polytope has volume $8$, but the defining monomials generate a sublattice of index $4$ in $\ZZ ^3$; we may thus work instead with a polytope of volume $2$ (depicted at right) and observe a significant speedup.
  In this setting, the Hodge numbers of $PH^{2}(\XFq)$ are $(1,4,1)$.
  With our current implementation  we  computed $Q(t)$ in \texttt{4s}.
  We present the running times for other $p$ in Table~\ref{table:k3c2f2}, where the memory footprint was about 52MB.
\par}%
Alternatively, one could try to use \textsc{Magma} \cite{magma} to confirm $Q(t)$.
Unfortunately, \textsc{Magma} is only able to confirm the linear coefficient:
\begin{verbatim}
      > C2F2 := HypergeometricData([6,12], [1,1,1,2,3]);
      > EulerFactor(C2F2, 2^10 * 3^6, 2^15 - 19: Degree:=1);
      1 + 20508*$.1 + O($.1^2)
\end{verbatim}
\begin{table}[h]
  \begin{tabular}{lr|lr|lr}
    $p$ & time & $p$ & time & $p$ & time \\
    \hline
    $2^{8} - 5$  & \texttt{0.20s} & $2^{13} - 1$ &  \texttt{1.12s} & $2^{18} - 5$ & \texttt{4m54s}  \\
    $2^{9} - 3$  & \texttt{0.23s} & $2^{14} - 3$ &  \texttt{2.08s} & $2^{19} - 1$ & \texttt{9m46s}  \\
    $2^{10} - 3$ & \texttt{0.29s} & $2^{15} - 19 $& \texttt{4.00s} & $2^{20} - 3$ & \texttt{19m32s}  \\
    $2^{11} - 9$ & \texttt{0.41s} & $2^{16} - 15$ & \texttt{1m11s} & $2^{21} - 9$ &  \texttt{38m58s}  \\
    $2^{12} - 3$ & \texttt{0.64s} & $2^{17} - 1$ &  \texttt{2m30s} & $2^{22} -3 $ & \texttt{1h18m}  \\
  \end{tabular}
  \caption{Running times for Example~\ref{example:c2f2}.}
  \label{table:k3c2f2}
\end{table}
\end{example}

\begin{example}
  \label{example:k3nonwps}
  Consider the closure $\XFq$ in $\PP_\Delta$ (which in this case is not a weighted projective space) of the affine surface defined by the Laurent polynomial
  \begin{multline*}
    3 x + y + z + x^{-2} y^{2} z + x^{3} y^{-6} z^{-2} + 3 x^{-2} y^{-1} z^{-2}\\
    -2 - x^{-1} y - y^{-1} z^{-1} - x^{2} y^{-4} z^{-1} - x y^{-3} z^{-1};
  \end{multline*}
  it is a K3 surface of geometric Picard rank 6, and the Hodge numbers of $PH^{2}(\XFq)$ are $(1,14,1)$.
  For $p = 2^{15} - 19$,
  using our current implementation, in \texttt{6m20s} we obtain
  the ``interesting'' factor of $Z(\XFq, t)$:
  \begin{equation*}
    \begin{aligned}
      p Q(t/p) =  (1 - t) \cdot (1 + t) 
      & \cdot (p + 33305t^{1} + 1564t^{2} - 14296t^{3} - 11865t^{4} \\
      &+ 5107t^{5} + 27955t^{6} + 25963t^{7} + 27955t^{8} + 5107t^{9}
    \\&- 11865t^{10} 
      - 14296t^{11} + 1564t^{12} + 33305t^{13} + p t^{14}).
    \end{aligned}
  \end{equation*}
  {\makeatletter
    \let\par\@@par
    \par\parshape0
    \everypar{}
    \begin{wrapfigure}{r}{0.28\textwidth}
      \vspace{-1.5em}
      \tcbox[colback=white, top=2pt,left=0pt,right=0pt,bottom=2pt, boxrule=0.2mm, on line]{
        \includegraphics[width=0.25\textwidth]{\polygons/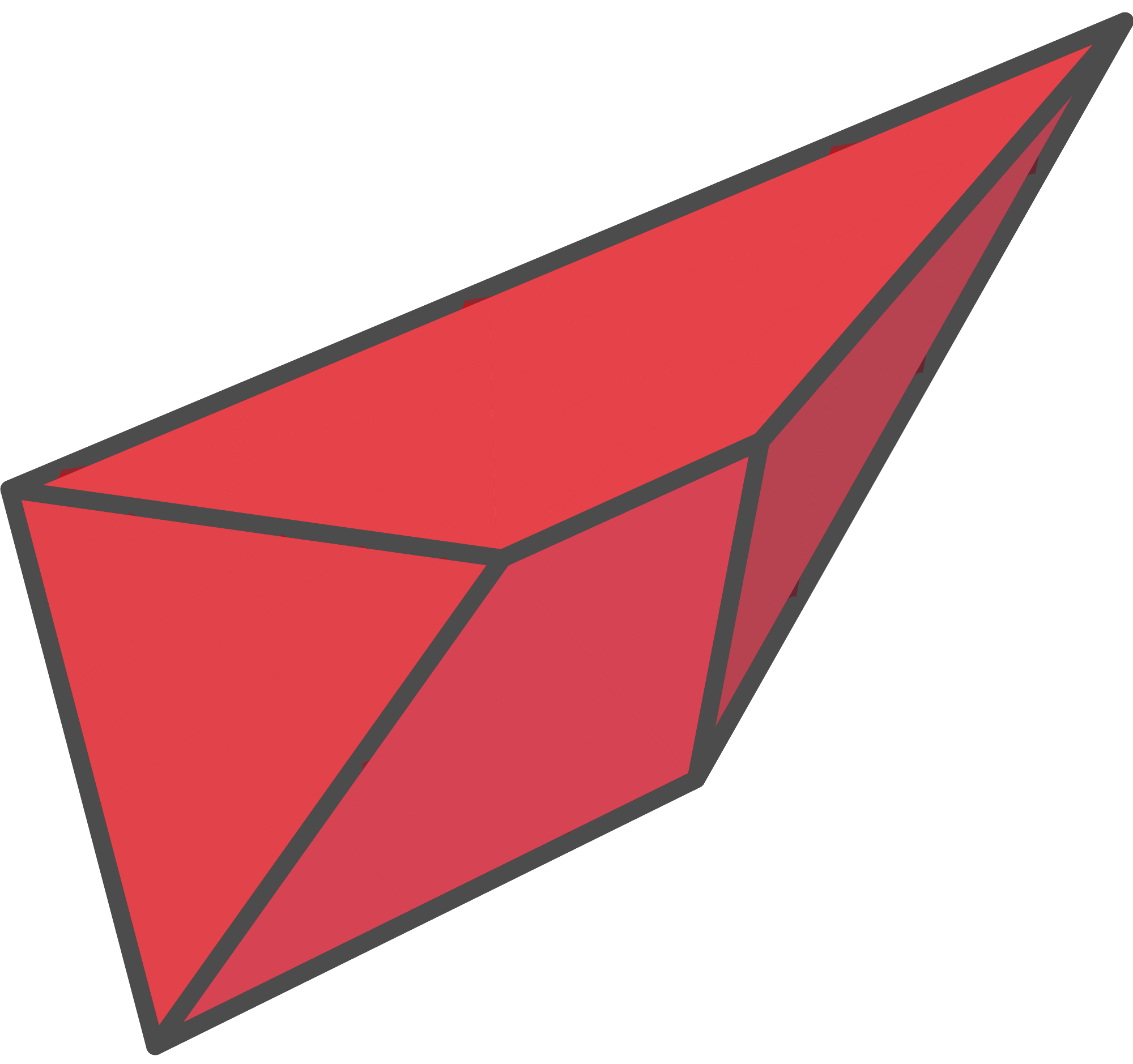}
      }
    \end{wrapfigure}
    \noindent
    We present the running times for other $p$ in Table~\ref{table:k3nonwps}, where the peak memory usage was about 144MB.

    The vertices of the associated polytope correspond to the first six terms displayed; the remaining terms are interior points.
    We depict this polytope of volume 8 at right.

    We know of no previous algorithm that can compute $Z(\XFq, t)$ for $p$ in this range. The defining polynomial is ``dense'' from the point of the Sperber--Voight  algorithm \cite{sperber-voight-13}, which is based on Dwork cohomology
    and scales
  \par}%
  \noindent with the number of monomials away from the vertices of the Newton polytope.

  \begin{table}[h]
    \begin{tabular}{lr|lr|lr}
      $p$ & time & $p$ & time & $p$ & time \\
      \hline
      $2^7 - 1$   & \texttt{6.46s} & $2^{10} - 3$ &  \texttt{18.93s} & $2^{13} - 1$ & \texttt{1m46s}  \\
      $2^{8} - 5$ & \texttt{9.50s} & $2^{11} - 9$ &  \texttt{31.34s} & $2^{14} - 3$ & \texttt{3m24s}  \\
      $2^{9} - 3$ & \texttt{12.64s} & $2^{12} - 3$ &  \texttt{56.23s} & $2^{15} - 19$ & \texttt{6m20s}  \\
    \end{tabular}
    \caption{Running times for Example~\ref{example:k3nonwps}.}
    \label{table:k3nonwps}
  \end{table}
\end{example}

\vspace{-36pt}
\begin{example}
  \label{example:k3dense}
  Let $\XFq$ be the smooth projective surface in $\PP^3$ defined by the fully dense, randomly chosen quartic polynomial
  \begin{multline*}
    -9 x^4 - 10 x^3 y - 9 x^2 y^2 + 2 x y^3 - 7 y^4 + 6 x^3 z + 9 x^2 y z - 2 x y^2 z + 3 y^3 z \\
    + 8 x^2 z^2 + 6 y^2 z^2  + 2 x z^3 + 7 y z^3 + 9 z^4 + 8 x^3 w + x^2 y w - 8 x y^2 w \\
    - 7 y^3 w
    + 9 x^2 z w
    - 9 x y z w + 3 y^2 z w - x z^2 w - 3 y z^2 w + z^3 w - x^2 w^2 \\
    - 4 x y w^2 - 3 x z w^2 + 8 y z w^2 - 6 z^2 w^2 + 4 x w^3 + 3 y w^3 + 4 z w^3 - 5 w^4;
  \end{multline*}
then $\Delta$ is the 3-simplex (tetrahedron) of volume $32/3 \approx 10.66$.
  For this example, we have $PH^{2} (\XFq) \simeq H^3(\PP^3 \backslash \XFq)$, the Hodge numbers are $(1,19,1)$, and
  \begin{equation*}
    Z(\XFq, t)^{-1} =  (1 - t) (1 - p t)  (1 - p^2 t) Q(t)
  \end{equation*}
  where $\deg Q(t) = 21$.
  For $p = 2^{15} - 19$, we obtain
  \begin{align*}
    p Q(t/p) =
      &(1 + t)
       \bigl ( p \!-\! 53159t^{1} \!+\! 10023t^{2} \!-\ 3204t^{3} \!+\! 49736t^{4} \!-\! 56338t^{5} \!+\! 43086t^{6} \\ &\!-\! 48180t^{7} \!+\! 44512t^{8} \!-\! 42681t^{9} \!+\! 47794t^{10} 
     \!-\! 42681t^{11} \!+\! 44512t^{12} \!-\! 48180t^{13} \\ &\!+\! 43086t^{14}
     \!-\! 56338t^{15} \!+\! 49736t^{16} \!-\! 3204t^{17} \!+\! 10023t^{18} \!-\! 53159t^{19} \!+\! p t^{20} \bigr)
  \end{align*}
using the \emph{controlled AKR algorithm} in \texttt{38m27s}; our current implementation takes roughly the same time. 
   We present the running times for other $p$ in Table~\ref{table:k3dense}. 
   The memory footprint was about 230MB.
   
  Unfortunately, the \emph{deformation method} is not suitable for dense quartics with $p$  in this range.
  For example, for $p = 31$ the running time was \texttt{2h8m} and its memory footprint was around 7GB, and both time and space should scale linearly with $p$.
  \begin{table}[h]
    \begin{tabular}{lr|lr|lr}
      $p$ & time & $p$ & time & $p$ & time \\
      \hline
      $2^7 - 1$ & \texttt{25.41s} & $2^{10} - 3$ &  \texttt{1m30s} & $2^{13} - 1$ & \texttt{9m26s}  \\
      $2^{8} - 17$ & \texttt{37.73s} & $2^{11} - 9$ &  \texttt{2m37s} & $2^{14} - 3$ & \texttt{18m42s}  \\
      $2^{9} - 3$ & \texttt{55.82s} & $2^{12} - 3$ &  \texttt{4m50s} & $2^{15} - 19$ & \texttt{36m29s}  \\
    \end{tabular}
    \caption{Running times for Example~\ref{example:k3dense}.}
    \label{table:k3dense}
  \end{table}
\end{example}

\vspace{-36pt}
\section{Calabi--Yau threefolds}

We next consider Calabi--Yau threefolds. Unlike for K3 surfaces, the middle Betti numbers of Calabi--Yau threefolds are not \emph{a priori} bounded; the largest value of which we are aware is 984 (found in  \cite{kreuzer-skarke-00}).

A common example is a smooth quintic surface in $\PP^4$.
Again, additional constructions arise from generic hypersurfaces in weighted projective spaces, of which there are $\num{7555}$ in all, or more generally from toric varieties corresponding to reflexive 4-dimensional polytopes, of which there are $\num{473800776}$ in all \cite{kreuzer-skarke-00}.

In all of the following examples, we worked modulo $p^6$ in order to obtain $Q(t)$ and our memory footprint ranged between 100MB and 270MB.
\begin{example}
  \label{example:3folddwork}
    Consider the projective quintic threefold $\XFq \subset \PP^3 _{\FF_p}$ defined by
  $$ x_0 ^5 + x_1 ^5 + x_2 ^5 + x_3 ^5 + x_4 ^ 5 + x_0 x_1 x_2 x_3 x_4 = 0;$$
  it is a member of the Dwork pencil. We have
  $$ Z(\XFq, t) = \frac{R_1(p t)^{20} R_2(p t) ^{30} Q(t)}{ (1 - t) (1 - p t) (1 - p^2 t) (1 - p^3 t) } $$
  where $R_1$ and $R_2$ are the numerators of the zeta functions of certain curves given by a formula of Candelas--de la Ossa--Rodriguez Villegas \cite{candelas-ossa-villegas-01}.
  
  As it is presented, we would work with $\PP _\Delta = \PP^4$ where $\Delta$ is the $4$-simplex of volume $625/24$.
  As in Example~\ref{example:k3dwork}, the monomials of the equation generate a sublattice of index $5^3$ in $\ZZ ^4$,
  so we may instead work with a polytope whose volume is smaller by a factor of $5^3$. 
  For $p =2^{20} - 3$, we compute the ``interesting'' factor
  $$Q(t) = 1  -1576492860 t^1 + 2672053179370 p t^2  -1576492860 p^3 t^3 + p^6 t^4$$
  in \texttt{11m18s}; if we instead had tried to apply the \emph{controlled AKR algorithm} to compute $Q(t)$ (and not the other factors) we extrapolate that it would take us at least 120 days.
  We present the running times for other $p$ in Table~\ref{table:3folddwork}.
  
  Since this is a ``small'' perturbation of the Fermat threefold,
  we again attempted to confirm these results using the \emph{deformation method}; however, this was again hampered by the fact that the Pancratz--Tuitman implementation works in $\PP_\Delta$ instead of $\PP^3$.
For $p=7$, it took \texttt{5h4m} and its memory footprint was around 12GB.
    \begin{table}[h]
    \begin{tabular}{lr|lr|lr}
      $p$ & time & $p$ & time & $p$ & time \\
      \hline
      $2^{8} - 5$  & \texttt{0.73s} & $2^{13} - 1$ &  \texttt{6.41s} & $2^{18} - 5$ & \texttt{2m50s}  \\
      $2^{9} - 3$  & \texttt{0.77s} & $2^{14} - 3$ &  \texttt{11.61s} & $2^{19} - 1$ & \texttt{5m38s}  \\
      $2^{10} - 3$ & \texttt{0.80s} & $2^{15} - 19 $& \texttt{21.98s} & $2^{20} - 3$ & \texttt{11m18s}  \\
      $2^{11} - 9$ & \texttt{2.54s} & $2^{16} - 15$ & \texttt{43.07s} & $2^{21} - 9$ &  \texttt{22m41s}  \\
      $2^{12} - 3$ & \texttt{3.80s} & $2^{17} - 1$ &  \texttt{1m25s} & $2^{22} -3 $ & \texttt{52m37s}  \\
    \end{tabular}
    \caption{Running times for Example~\ref{example:3folddwork}.}
    \label{table:3folddwork}
  \end{table}
\end{example}
\begin{example}
  \label{example:3fold1111}
  Let $\XFq$ be the threefold defined by
  $$
  x_{0}^{8} + x_{1}^{5} x_{2} + x_{0}^{2} x_{1}^{2} x_{2} x_{3} + x_{1} x_{2}^{3} x_{3} + x_{1}^{2} x_{3}^{3} + x_{0} x_{1} x_{2} x_{3} x_{4} + x_{2} x_{3} x_{4}^{2}
  $$
  in the weighted projective space $\PP(1, 14, 18, 20, 25)$.
  The Newton polytope has volume $11/3 \approx 3.67$; by changing the lattice we may instead work with a polytope of volume $1/3 \approx 0.33$. In this setting, the Hodge numbers of $PH^3 (\XFq)$ are $(1,1,1,1)$.
  
  For $p = 2^{20} -3$, we compute the ``interesting'' factor of $Z(\XFq, t)$
  \begin{equation*}
    1 - 618297672 t^1 + 390956360946 p t^2 -618297672 p^3 t^3  + p^6 t^4
  \end{equation*}
  in \texttt{32m33s}. 
    We present the running times for other $p$ in Table~\ref{table:3fold1111}.
  \begin{table}[h]
    \begin{tabular}{lr|lr|lr}
      $p$ & time & $p$ & time & $p$ & time \\
      \hline
      $2^{8} - 5$  & \texttt{1.90s} & $2^{13} - 1$ &  \texttt{18.2s} & $2^{18} - 5$ & \texttt{8m0s}  \\
      $2^{9} - 3$  & \texttt{1.96s} & $2^{14} - 3$ &  \texttt{32.9s} & $2^{19} - 1$ & \texttt{16m8s}  \\
      $2^{10} - 3$ & \texttt{2.06s} & $2^{15} - 19 $& \texttt{1m6s} & $2^{20} - 3$ & \texttt{32m33s}  \\
      $2^{11} - 9$ & \texttt{7.48s} & $2^{16} - 15$ & \texttt{2m4s} & $2^{21} - 9$ &  \texttt{1h5m}  \\
      $2^{12} - 3$ & \texttt{10.9s} & $2^{17} - 1$ &  \texttt{4m3s} & $2^{22} -3 $ & \texttt{2h23m}  \\
    \end{tabular}
    \caption{Running times for Example~\ref{example:3fold1111}.}
    \label{table:3fold1111}
  \end{table}
  \vspace{-24pt}
\end{example}

\begin{example}
  \label{example:3fold1221}
  Let $\XFq$ be the threefold defined by
  $$
  x_{1}^{7} + x_{0}^{5} x_{1} x_{2} + x_{0}^{2} x_{1}^{2} x_{2} x_{3} + x_{0}^{4} x_{2} x_{4} + x_{0} x_{2}^{3} x_{3} + x_{0}^{2} x_{3}^{3} + x_{0} x_{1} x_{2} x_{3} x_{4} + x_{2} x_{3} x_{4}^{2}
  $$ in the weighted projective space $\PP(10, 11, 16, 19, 21)$.
  Again, by choosing the right lattice, we reduce the volume of the Newton polytope from $55/12 \approx 4.58$ to
  $11/24 \approx 0.46$, and the Hodge numbers of $PH^3 (\XFq)$ are $(1,2,2,1)$.
  For $p = 2^{20} -3$, we computed the ``interesting'' factor of $Z(\XFq, t)$
  \begin{multline*}
    1 - 2068001468 t^1 + 3449674041773 p t^2 - 3772715295733197 p^2 t^3 \\
    +3449674041773 p^4 t^4 - 2068001468 p^6 t^5 + p^9 t^6
  \end{multline*}
  in \texttt{2h10m}. We present the running times for other $p$ in Table~\ref{table:3fold1221}.
\begin{table}[h]
    \begin{tabular}{lr|lr|lr}
      $p$ & time & $p$ & time & $p$ & time \\
      \hline
      $2^{8} - 5$  & \texttt{4.47s} & $2^{13} - 1$ &  \texttt{1m8s} & $2^{18} - 5$ &  \texttt{32m25s}  \\
      $2^{9} - 3$  & \texttt{4.60s} & $2^{14} - 3$ &  \texttt{2m8s} & $2^{19} - 1$ & \texttt{1h5m}  \\
      $2^{10} - 3$ & \texttt{4.96s} & $2^{15} - 19 $& \texttt{4m6s} & $2^{20} - 3$ & \texttt{2h10m}  \\
      $2^{11} - 9$ & \texttt{25.8s} & $2^{16} - 15$ & \texttt{8m18s} & $2^{21} - 9$ &  \texttt{4h17m}  \\
      $2^{12} - 3$ & \texttt{39.1s} & $2^{17} - 1$ &  \texttt{16m31s} & $2^{22} -3 $ & \texttt{9h33m}  \\
    \end{tabular}
    \caption{Running times for Example~\ref{example:3fold1221}.}
    \label{table:3fold1221}
  \end{table}
\end{example}

\begin{example}
  \label{example:3foldnonwps}
  Let $\XFq$ be the closure in $\PP_\Delta$ (which is not a weighted projective space)
of the threefold defined by the Laurent polynomial
  \begin{equation*}
    x y z^2 w^3 + x + y + z - 1 + y^{-1} z^{-1} + x^{-2} y^{-1} z^{-2} w^{-3} = 0.
\end{equation*}
Choosing the correct lattice reduces the volume of the Newton polytope from $9/8 \approx 1.12$
to $3/8 \approx 0.38$, and the Hodge numbers of $PH^{3}(\XFq)$ are $(1,2,2,1)$.
For $p = 2^{20} -3$, we computed the ``interesting'' factor of $Z(\XFq, t)$
  \begin{multline*}
    (1 + 718 p t + p^3 t^2)  
    \cdot (1 + 1188466826 t^1 + 1915150034310 p t^2 +  1188466826 p^3 t^3 + p^6 t^4)
  \end{multline*}
  in \texttt{1h15m}. We present the running times for other $p$ in Table~\ref{table:3foldnonwps}.
  \begin{table}[h]
    \begin{tabular}{lr|lr|lr}
      $p$ & time & $p$ & time & $p$ & time \\
      \hline
      $2^{8} - 5$  & \texttt{2.74s} & $2^{13} - 1$ &  \texttt{39.28s} & $2^{18} - 5$ & \texttt{18m34s}  \\
      $2^{9} - 3$  & \texttt{2.80s} & $2^{14} - 3$ &  \texttt{1m13s} & $2^{19} - 1$ & \texttt{38m8s}  \\
      $2^{10} - 3$ & \texttt{3.00s} & $2^{15} - 19 $& \texttt{1m21s} & $2^{20} - 3$ & \texttt{1h15m}  \\
      $2^{11} - 9$ & \texttt{14.86s} & $2^{16} - 15$ & \texttt{4m45s} & $2^{21} - 9$ &  \texttt{2h32m}  \\
      $2^{12} - 3$ & \texttt{22.32s} & $2^{17} - 1$ &  \texttt{9m12s} & $2^{22} -3 $ & \texttt{5h39m}  \\
    \end{tabular}
    \caption{Running times for Example~\ref{example:3foldnonwps}.}
    \label{table:3foldnonwps}
  \end{table}
\end{example}

\vspace{-30pt}
\section{Cubic fourfolds}

For our final example, we consider a cubic fourfold. For $X$ a smooth cubic fourfold in $\PP^5$,  $\dim H^4(X) = 23$ and the Hodge numbers are $(0, 1, 21, 1, 0)$.

In this example,
we worked modulo $p^6$ in order to obtain $Q(t)$.

\begin{example}
  \label{example:4fold}
  Let $\XFq$ be the smooth projective cubic fourfold in $\PP^5_{\FF_p}$ defined by
\begin{equation*}
  x_0^3 + x_1^3 + x_2^3 + (x_0 + x_1 + 2x_2)^3 + x_3^3 + x_4^3  + x_5^3 + 2(x_0+x_3)^3 + 3(x_1+x_4)^3 + (x_2+x_5)^3;
\end{equation*}
it is nondegenerate in $\PP^5$. For $p=31$, in \texttt{21h31m} we computed
  \begin{equation*}
    Z(\XFq, t) ^{-1} = (1 -  t) (1 - p t) (1 - p^2 t) (1-p^3 t) (1 - p^4 t) Q(t)
  \end{equation*}
  where the ``interesting'' factor is an irreducible Weil polynomial given by
  \begin{align*}
    p Q(t/p^{2}) = &\,
    p \!-\! 7t^1 \!+\! 21t^2 \!-\! 52t^3 \!-\!8t^4 \!-\! 28t^5 \!+\! 21t^6 \!+\! 35t^7 
    \!+\! 39t^9 \!+\! 62t^{10}\!+\! 23t^{11}  \\ 
    &\!+\! 62t^{12} \!+\! 39t^{13} \!+\! 35t^{15} \!+\! 21t^{16} \!-\!28t^{17} \!-\! 8t^{18} \!-\! 52t^{19} \!+\! 21 t^{20} \!-\! 7t^{21} \!+\! pt^{22};
  \end{align*}
the coefficient of $t^1$ may be confirmed independently by counting $\XFq(\FF_p)$ using the Sage function \verb+count_points+.
  For $p = 127$ the running time was \texttt{23h15m} and for $p = 499$ it was \texttt{24h55m};
  in both cases, the ``interesting'' factor is an irreducible Weil polynomial.
  In these computations, the memory footprint was around 36.5GB.

  In dimension 4, the bottleneck seems to be the linear algebra required to set up  controlled reduction.  For $p=31$, more than half of the runtime (\texttt{15h32m}) is spent solving a linear problem of size $\num{15504} \times \num{37128}$ modulo $p^6$.
  With careful handling of this step (e.g., avoiding Hensel lifts) we would expect a significant speedup.
 
Note that the defining equation for $\XFq$ is quite sparse. To assess the effect of this sparsity, as well as to cross-check the answer, we recomputed $Z(\XFq,t)$ after applying a random linear change of variables to obtain a dense defining equation. For $p=31$, in  \texttt{27h55m} and using about 41GB we obtained the same value for $Z(\XFq,t)$ as above.
\end{example}

Recall from the introduction that a cubic fourfold is \emph{coplanar} if it is defined by an expression  $\sum_{i=1}^{10} a_i^3$ in which each $a_i$ is a linear form and some four of the $a_i$ are linearly dependent. Ranestad--Voisin \cite{ranestad-voisin-17} show that the Zariski closure $D_{copl}$ of the coplanar locus is a divisor on the moduli space of cubic fourfolds. Example~\ref{example:4fold} is a non-special coplanar cubic fourfold: the existence of a primitive codimension-2 cycle class would imply\footnote{While it is not needed here, the Tate conjecture for ordinary cubic fourfolds is known \cite{levin}.}  that $p Q(t/p^2)$ has a cyclotomic factor. This shows (modulo detailed analysis of the algorithm) that $D_{copl}$ is not a Noether--Lefschetz divisor.

\bibliographystyle{alpha}
\bibliography{biblio-compressed}

\end{document}